\documentclass[leqno,11pt]{amsart}
\newtheorem{thm}{Theorem}[section]
\newtheorem{cor}[thm]{Corollary}

\theoremstyle{definition}

\theoremstyle{remark}

\newtheorem{remark}[thm]{Remark}
\newtheorem{notation}[thm]{Notation}

\newlength{\intwidth}
\DeclareRobustCommand{\fpint}[2]
   {\mathop{%
      \text{%
              \settowidth{\intwidth}{$\int$}%
              \makebox[0pt][l]{\makebox[\intwidth]{$-$}}%
              $\int_{#1}^{#2}$
           }
           }
   }

\title[Degenerating behavior of Green's function]{Degenerating behavior of Green's function}
\author[F. Peherstorfer]{F. Peherstorfer$^1$}
\thanks{$^1$The manuscript was prepared by the author in the two months preceding his passing away 
in November 2009. The manuscript remained unsubmitted and is not published elsewhere 
(submitted by P. Yuditskii and I. Moale).}

\begin{document}

\begin{abstract}
Let the unions of real intervals $I = \cup_{j = 1}^l [a_{2 j -1},a_{2j}],$
$a_1 < ... < a_{2 l},$ and $I_n = \cup_{k = 1}^m [B_{k,n}, C_{k,n}]$ be such
that $\cap_{k = 1}^{\infty} [B_{k,n},C_{k,n}] = \{ c_k \}$ for $k = 1,...,m$
and ${\rm dist}(E,I_n) \geq const > 0.$ We show how to express asymptotically
the Green's function $\phi(z,\infty,E \cup I_n)$ of $E \cup I_n$ at $z = \infty$
in terms of the Green's function $\phi(z,\infty,E)$ and $\phi(z,c_k,E).$ The
formula yields immediately asymptotics for $\phi^n(z,\infty,E \cup I_n)$
with respect to $n$ which are important in many problems of approximation
theory. Another consequence is an asymptotic representation of $cap(E \cup I_n)$
in terms of $cap(E)$ and $\phi(z,c_k,E)$ and of the harmonic measure
$\omega(\infty, E_j,E \cup I_n).$
\end{abstract}

\maketitle

Let $E = \bigcup_{j=1}^{l} [a_{2j-1},a_{2j}],$ $-\infty < a_1 <
a_2 < ... < a_{2l} < \infty,$ be a union of $l$ disjoint
intervals, put $H(x) = \prod_{j=1}^{2l} (x - a_j)$ and let
$\phi(z,z_0,E)$ be a so-called complex Green's function, that is,
a mapping which maps $\bar{\mathbb C} \backslash E$ onto the
exterior of the unit circle, which has a simple pole at $z = z_0
\in \bar{\mathbb C} \backslash E$ and satisfies $|\phi(z,z_0)| \to
1$ for $z \to x \in E$; or in other words $\log |\phi|$ is
the Green's function. It is known, see e.g. \cite{Wid}, that
\begin{equation*}
     \phi(z,\infty,E) := \phi(z,\infty) =
     \exp \left( \int_{a_{2l}}^{z} r_{\infty}(\xi) \frac{d\xi}{\sqrt{H(\xi)}} \right)
\end{equation*}
where $r_{\infty}(\xi) = \xi^{l-1} + ...$ is the unique
polynomial such that
\begin{equation}\label{t1}
    \int_{a_{2j}}^{a_{2j+1}} r_{\infty}(\xi) \frac{d\xi}{\sqrt{H(\xi)}} =
    0 {\rm \ for \ } j = 0, ..., l-1
\end{equation}
and that for $x \in {\mathbb R} \backslash E$
\begin{equation}
   \phi(z,x_0,E) := \phi(z,x_0) =
   \exp \left( \int_{a_{2l}}^{z}
   \frac{r_{x_0}(\xi)}{\xi - x_0} \frac{d\xi}{\sqrt{H(\xi)}} \right)
\end{equation}
where $r_{x_0} \in {\mathbb P}_{l-1}$ is such that
\begin{equation}\label{t2}
   r_{x_0}(x_0) = - \sqrt{H(x_0)}
\end{equation}
and
\begin{equation}\label{t3}
   \fpint{a_{2j}}{a_{2j+1}} \frac{r_{x_0}(\xi)}{\xi - x_0}
   \frac{d\xi}{\sqrt{H(\xi)}} = 0 {\rm \ for \ } j = 1, ... , l-1.
\end{equation}
The so-called capacity of $E$ is given by
\begin{equation}\label{t4}
     cap (E) = \lim\limits_{z \to \infty} |\frac{z}{\phi(z,\infty,E)}|
\end{equation}
By $\omega(z,K,E)$ we denote the harmonic measure of the set
$K \subset E$ with respect to $E.$ Recall that
\begin{equation}\label{t5}
    \omega(\infty,[a_{2 j - 1},a_{2 j}],E) =
    \int_{a_{2 j - 1}}^{a_{2 j}} |r_\infty(x)| \frac{dx}{\sqrt{|H(x)|}}
\end{equation}

\begin{thm}\label{first}
(Asymptotic representation of the Green's function in the
degenerate case) Let $l,m \in \mathbb N,$
$E = \bigcup_{j=1}^{l} [a_{2j-1},a_{2j}]$ and for a strictly monotone
increasing subsequence $(n)$ of $\mathbb N$ let
$I_n = \bigcup_{k=1}^{m} [B_{k,n}, C_{k,n}]$ with
$dist(E, I_n) \geq const > 0$  and $\bigcap_{n=1}^{\infty}
[B_{k,n}, C_{k,n}] = \{ c_k \}$ and let $\omega_{k,n} =$
\\$\omega(\infty, [B_{k,n},C_{k,n}],E \cup I_n)$ with
$\omega_{k,n} \underset {n \to \infty} \longrightarrow 0$ for $k =
1,...,m.$ Then uniformly on compact subsets of ${\mathbb C}
\backslash ( E \cup \{ c_1,...,c_m \})$
\begin{equation}\label{last}
   \phi (z, \infty, E \cup I_n) =
   \frac{\phi (z, \infty, E)}
   {\prod\limits_{k=1}^{m}(\phi (z, c_k, E))^{\omega_{k,n}}}
   (1 + O(\max\limits_{1 \leq k \leq m} \omega_{k,n}^2))
\end{equation}
\end{thm}

\begin{proof}
For abbreviation we put
\begin{equation*}
   \tilde{E}_n = E \cup I_n, \ H_{\tilde{E},n}(x) = H(x) \tilde{H}_n(x)
\end{equation*}
where
\begin{equation*}
   H(x) = \prod_{j=1}^{2l} (x-a_j), \ \tilde{H}_n(x) = \prod_{k=1}^{m} (x-B_{k,n})(x-C_{k,n})
\end{equation*}
By \cite{Wid} we know that  there exists a monic polynomial $r_{\tilde{E},n}(x)$
of degree $l+m-1$ such that
\begin{equation*}
   \phi(z, \infty, \tilde{E}_n) =
   \exp \left( \int \frac{n r_{\tilde{E},n}(x)}{\sqrt{H_{\tilde{E},n}(x)}} dx \right)
\end{equation*}
Now let us represent $r_{\tilde{E},n}$ in the form
\begin{equation*}
\begin{aligned}
    r_{\tilde{E},n}(x)
    & = \tilde{s}_n(x) \prod_{k=1}^{m} (x-c_k) + t_n(x)\\
    & = r_{E}(x) \prod_{k=1}^{m} (x-c_k) + s_n(x) \prod_{k=1}^{m} (x-c_k) + t_n(x) \\
\end{aligned}
\end{equation*}
where $r_E \in {\mathbb P}_{l-1}$ is the monic polynomial
associated with the Green's function $\phi(z, \infty, E)$,
$s_n(x) = \tilde{s}_n(x) - r_E(x) \in {\mathbb P}_{l-2}$
and $t_n \in {\mathbb P}_{m-1}.$ Putting
\begin{equation*}
    L_n = \max_{1 \leq k \leq n} |C_{k,n} - B_{k,n}|
\end{equation*}
we have on compact subsets of ${\mathbb C} \backslash \{ c_1,...,c_m \}$
\begin{equation*}
    \frac{1}{\sqrt{H_{\tilde{E},n}(x)}} =
    \frac{1}{\prod\limits_{k=1}^{m} (x - c_k) \sqrt{H(x)}} + O(L_n)
\end{equation*}
and thus
\begin{equation}\label{tilde}
\begin{aligned}
   \frac {r_{\tilde{E},n}(x)} {\sqrt{H_{\tilde{E},n}(x)}} dx =
     & \frac{r_{E}(x)}{\sqrt{H(x)}} dx +
 \frac{s_{n}(x)}{\sqrt{H(x)}}dx \\
     & \ + \frac{t_n(x)}{\prod\limits_{k=1}^{m}(x-c_k)}
 \frac{dx}{\sqrt{H(x)}} + O(L_n)
\end{aligned}
\end{equation}

Partial fraction expansion gives, recall $t_n \in {\mathbb P}_{m-1},$
\begin{equation}\label{xto1}
    \frac{t_n(x)}{\prod\limits_{k=1}^{m} (x-c_k)} =
    \sum\limits_{k=1}^{m} \frac{\lambda_{k,n}}{x-c_k}
\end{equation}
To determine the $\lambda_{k,n}$'s we integrate both sides in
\eqref{tilde} counterclockwise around a circle with center $c_k$
and fixed small radius $\varepsilon.$ Taking into consideration
the facts, that the limiting values
$\pm \sqrt{H_{\tilde{E},n}(x)} := \lim\limits_{\substack{z \to x^{\pm} \\ x \in \tilde{E},n}} \frac{1}{\sqrt{H(z)}}$
from the upper- and lower half plane satisfy $+ \sqrt{H_{\tilde{E},n}(x)} = -
\sqrt{H_{\tilde{E},n}(x)}$ and that
\begin{equation*}
   + \frac{1}{\sqrt{H_{\tilde{E},n}(x)}} =
   \frac{sgn \ r_{\tilde{E},n}(x)}{i \sqrt{|H_{\tilde{E},n}(x)|}}
   {\rm \ \ \ \ \ for \ } x \in \tilde{E},n,
\end{equation*}
we obtain for the LHS by shrinking the circle
to the interval $[B_{k,n},C_{k,n}]$ that
\begin{equation}
   \oint \frac{ r_{\tilde{E},n}(\xi) }{ \sqrt{H_{\tilde{E},n}(\xi)} } d\xi =
   \frac{2}{i} \int_{C_{k,n}}^{B_{k,n}}
   \frac{|r_{\tilde{E},n}(x)|}{ \sqrt{|H_{\tilde{E},n}(x)|} } dx =
   2 i \pi \omega_{k,n}
\end{equation}
On the other hand the RHS from \eqref{tilde} becomes, note that
all the terms at the RHS are analytic on the circle up to
$1/(x-c_k),$
\begin{equation*}
\begin{aligned}
      \oint \frac{ r_{\tilde{E},n}(x) }{ \sqrt{H_{\tilde{E},n}(x)} } dx
            & = \lambda_{k,n} \oint \frac{ 1 }{ x - c_k } \frac{ dx }{ \sqrt{H (x)} } + O(L_n)\\
            & = \frac{ \lambda_{k,n} 2 \pi i }{ \sqrt{H(c_k)} } + O(L_n)
\end{aligned}
\end{equation*}
that is,
\begin{equation}\label{xla2}
      \lambda_{k,n} = \omega_{k,n} \sqrt{H(c_k)} + O(L_n)
\end{equation}

Summarizing, on compact subsets of ${\mathbb C} \backslash \{ c_1,...,c_m \}$ we have
\begin{equation}\label{tilde0}
\begin{aligned}
   \frac {r_{\tilde{E},n}(x)} {\sqrt{H_{\tilde{E},n}(x)}} dx =
       & \frac{r_E(x)}{\sqrt{H(x)}} dx + s_n(x)
   \frac{dx}{\sqrt{H(x)}} \\
       & + \sum \omega_{k,n} \frac{\sqrt{H(c_k)}}{x-c_k} \frac{dx}{\sqrt{H(x)}} + O(L_n)
\end{aligned}
\end{equation}

Finally let us determine $s_n(x)$ asymptotically. We have by
\eqref{tilde}, \eqref{xto1} and \eqref{xla2} that
\begin{equation*}
\begin{aligned}
    -i \pi \omega_{\nu,n} = \int_{B_{\nu,n}}^{C_{\nu,n}}
    \frac{r_{\tilde{E},n}(x)}{ \sqrt{H_{\tilde{E},n}(x)} } dx
    = \int_{a_{2\nu}}^{a_{2\nu+1}} \frac{r_{\tilde{E},n}(x)}{ \sqrt{H_{\tilde{E},n}(x)} } dx\\
    = \int_{a_{2\nu}}^{a_{2\nu+1}} \left( s_n(x) + \sum_{k=1}^m
    \frac{\omega_{k,n} \sqrt{H(c_k)}}{ x - c_k} \right) \frac{dx}{\sqrt{H(x)}} + O(L_n)
\end{aligned}
\end{equation*}
Observing that
\begin{equation*}
\begin{aligned}
    \fpint{a_{2\nu}}{a_{2 \nu + 1}} \frac{r_{c_k}(\xi)}{x - c_k} \frac{d \xi}{\sqrt{H(\xi)}} =
    & \int_{a_{2 \nu}}^{a_{2 \nu + 1}}
      \frac{r_{c_k}(x) - r_{c_k}(c_k)}{x - c_k} \frac{dx}{\sqrt{H(x)}} \\
    & - \sqrt{H(c_k)} \fpint{a_{2 \nu}}{a_{2 \nu + 1}}
      \frac{1}{\xi - c_k} \frac{d \xi}{\sqrt{H(\xi)}}
\end{aligned}
\end{equation*}
and that by \eqref{t2}
\begin{equation}
   \fpint{a_{2 \nu}}{a_{2 \nu + 1}} \frac{1}{\xi - c_k}
   \frac{d \xi}{\sqrt{H(\xi)}}= i \pi \delta_{k \nu}
\end{equation}
it follows that for $\nu = 1, ... , l-1 $
\begin{equation}\label{4prim}
    \fpint{a_{2 \nu}}{a_{2 \nu+1}} \left( s_{n}(\xi) -
    \sum_{k=1}^{m} \omega_{k,n} q_{c_{k}}(\xi) \frac{d\xi}{\sqrt{H(\xi)}} \right) d \xi = O(L_n)
\end{equation}
Recall that $s_n(\xi) - \sum_{k = 1}^m \omega_{k,n} q_{c_k}(\xi) \in {\mathbb P}_{l - 2}.$
Since a polynomial $q \in {\mathbb P}_{l-2}$ satisfying
\begin{equation}\label{ex}
    \int_{a_{2k}}^{a_{2k+1}} q(x) \frac{dx}{\sqrt{H(x)}} = 0 \ {\rm for} \ k = 1,...,l-1
\end{equation}
is identically the zero polynomial it follows with the help of Cramer's rule that
\begin{equation*}
    s_n(x) - \sum_{k=1}^{m} \omega_{k,n} q_{c_{k,n}}(x) = O(L_n)
\end{equation*}
which gives by \eqref{tilde0} the assertion, if we are able to show that
\begin{equation}\label{o}
    L_n = O( \max\limits_{1 \leq k \leq n} \omega_{k,n}^2)
\end{equation}
If $c_k \in (a_1,a_{2l})$ let us consider
\begin{equation*}
    {\mathcal E}_{k,n} = [a_1, C_{k,n}] \cup [a_{2k+1},a_{2l}],
    \ {\rm respectively,} \
    {\mathcal E}_{k,\infty} = [a_1, c_{k}] \cup [a_{2k+1},a_{2l}]
\end{equation*}
and let $A_{k,n}$ and $A_{k,\infty}$ be the critical point of
$\ln \phi(x, \infty, {\mathcal E}_{k,n})$ and
$\ln \phi(x, \infty, {\mathcal E}_{k,\infty})$, respectively. Then it follows that
\begin{equation*}
    A_{k,n} \to A_{k, \infty} \in (c_k, a_{2k+1}).
\end{equation*}
Thus we obtain for $n \geq n_0$
\begin{equation*}
\begin{aligned}
   & \omega(\infty, [B_{k,n}, C_{k,n}], \tilde{E}_n) \geq
     \omega(\infty, [B_{k,n}, C_{k,n}], {\mathcal E}_{k,n})\\
   & = \int_{B_{k,n}}^{C_{k,n}}
     \frac{x - A_{k,n}}{\sqrt{(x - a_1)(x - a_{2l})(x - C_{k,n})(x - a_{2k + 1})}} dx \\
   & \geq const \int_{B_{k,n}}^{C_{k,n}}
     \frac{dx}{\sqrt{C_{k,n}-x}} = const \sqrt{C_{k,n} - B_{k,n}},
\end{aligned}
\end{equation*}
where $const > 0$ is independent of $n,$ and this proves \eqref{o}.

If $c_k \in {\mathbb R} \backslash [a_1, a_{2l}]$ then we consider
${\mathcal E}_{k,n} := [a_1, C_{k,n}] \supset \tilde{E}_n,$
respectively, $[B_{k,n}, a_{2l}]$ and the assertion follows even simpler as above.
\end{proof}

\begin{remark}
By conformal transplantation and the invariance property of the harmonic measure
it follows that the Theorem holds for much more general sets as circular slits
in the plane, etc. .
\end{remark}

\begin{cor}
As above let $\omega_{k,n} = \omega(\infty, [B_{k,n},C_{k,n}],
E \cup \bigcup_{k=1}^{m} [B_{k,n},C_{k,n}]).$
Then the following statements hold:
\begin{enumerate}
\item[a)] $ cap \ (E \cup \bigcup_{k=1}^{m} [B_{k,n},C_{k,n}])$
          $ = ( cap \ E ) \prod\limits_{k=1}^{m} |\phi(\infty, c_k, E)|^{\omega_{k,n}} ( cap \ E )
            (1 + O(\max \limits_{k} \omega_{k,n}^2 ))$
\item[b)] $ \omega (\infty, E_j, E \cup \bigcup_{k=1}^{m} [B_{k,n},C_{k,n}]) $
          $ = \omega (\infty, E_j, E) - \sum\limits_{k=1}^{m} \omega_{k,n} \omega(c_k, E_j, E)
            + O(\max \limits_{k} \omega_{k,n}^2) $
\end{enumerate}
\end{cor}

\begin{proof}
a) Follows immediately by the fact that for a compact set $K$ \\
$\lim\limits_{z \to \infty} |\frac{z}{\phi(z, \infty, K)}| = cap \ (K)$\\
b) follows by taking the $\log$ in \eqref{last} and using \eqref{t5}.
\end{proof}

\begin{notation}
a) Let $p_n(x) = a_n x^n + ...;$ by $lc(p_n) = a_n$ we denote the leading coefficient
of $p_n.$

b) $x_1,...,x_m \in E$ with $x_1 < x_2 < ... < x_m$ are called alternation points
of $f \in C[a,b]$ if $f(x_i) = \pm (-1)^i ||f||_{\infty, E}$ for $i = 1,...,m.$
By $\sharp A(f;K)$ we denote the number of alternation points of $f$ on $K.$
\end{notation}

\begin{thm}\label{four}
Let $l,m \in \mathbb N,$ $E = \bigcup_{j=1}^{l} [a_{2j-1},a_{2j}]$
and $I_n = \bigcup_{k=1}^{m} [B_{k,n}, C_{k,n}]$ with $dist(E,
I_n) \geq const > 0$ for a strictly monotone increasing
subsequence of $\mathbb N$ and $\bigcap_{k=1}^{\infty} [B_{k,n},
C_{k,n}] = \{ c_k \}.$

Suppose that there is a sequence of polynomials $(p_n)$ such that
$p_n \in {\mathbb P}_n$ has $n$ real zeros, $\min \{ |p_n(y_i)| :
p'_n(y_i) = 0 \} \geq 1,$ $||p_n||_{E \cup I_n} \leq 1$ and
$\sharp A(p_n ; I_{k,n}) = \nu_k + 1$ for $k = 1, ..., m$ and $n \in
\mathbb N.$ Then on compact subsets of
${\mathbb C} \backslash (E \cup \{ c_1,...,c_m \})$
$p_n$ has an asymptotic representation of the form
\begin{equation}\label{alt1}
    2 p_n(x) = \left( \psi_n(x) + \frac{1}{\psi_n(x)} \right) (1 + o(1))
\end{equation}
where
\begin{equation*}
    \psi_n(x) = \frac{\phi^n(x, \infty, E)}{\prod\limits_{k=1}^{m} \phi^{\nu_k}(x,c_k,E)}
\end{equation*}
Furthermore the norm of the monic polynomial $\hat{p}_n(x) = p_n(x) /lc(p_n)$
is given by
\begin{equation}\label{alt2}
   ||\hat{p}_n||_E = 1/lc(p_n) = 2 (cap \ E)^n \prod_{j = 1}^m \phi^{\nu_k}(c_j;\infty,E)(1+o(1))
\end{equation}
where $0 < r < 1,$ and concerning the alternation points on $E_\nu$ we have
\begin{equation*}
  \sharp A(p_n ; E_\nu) = n \omega(\infty, E_\nu, E) +
  \sum_{k=1}^{l-1} \omega(c_k, E_{\nu}, E)
  + O(\frac{1}{n}) \ {\rm for \ } \nu = 1, ..., l
\end{equation*}
\end{thm}

\begin{proof}
By the properties of $p_n$ it follows that $p_n^{-1}([-1,1]) = E \cup I_n$
which implies, see e.g. \cite{PehSte, Ran} that
\begin{equation*}
\begin{aligned}
    \phi(z, \infty, E \cup I_n)
    & = \left( p_n(z) - \sqrt{p_n^2(z) - 1} \right)^{\frac{1}{n}}\\
    & = \exp \left( {\frac{1}{n} \int_{a_{2l}}^{z} \frac{p_n^{'}(x)}{\sqrt{p_n^2(x) - 1}} dx} \right)
\end{aligned}
\end{equation*}
Thus
\begin{equation*}
  2 p_n(z) = \phi^n(z,\infty, E \cup I_n) +
  \frac{1}{\phi^n(z,\infty,E \cup I_n)}
\end{equation*}
and the representation \eqref{alt1} follows by Theorem \ref{first}.

Concerning relation \eqref{alt2} let us observe that, with the help of \eqref{alt1}
and $|\phi(z,\infty,E)| > 1$ on ${\mathbb C} \backslash E,$
\begin{equation*}
   lc(p_n) = \lim_{z \to \infty} \frac{p_n(z)}{z^n} =
   \lim_{z \to \infty}
   \left(  \frac{1}{2}
   \left( \frac{\phi(z,\infty,E)}{z}
   \right)^n
   \frac{1}{\prod_{k = 1}^m \phi^{\nu_k}(z,c_k,E)}
   \right)(1 + o(1))
\end{equation*}
\end{proof}

By the Alternation Theorem it follows immediately that many $L_\infty$-minimal
polynomials without or with restrictions as minimal polynomials on several intervals
satisfy the assumptions of Theorem \ref{four} with $\nu_k = 1,$ $k = 1,...,m,$
and thus have a representation of the form \eqref{alt1}. One of the remaining
challenging problems is to describe when $c_k$'s appear between two consecutive
$E_\nu$'s. For a solution of this problem in the case of minimal polynomials
on several intervals, see \cite{Pehmin}.


\begin{thebibliography}{99}

\bibitem{Pehdef} F. Peherstorfer, {\em Deformation of minimal polynomials},
J. Approx. Theory 111 (2001), 180-195.

\bibitem{Pehmin} F. Peherstorfer, {\em Asymptotic representation of minimal
polynomials on several intervals}, preprint.

\bibitem{PehSte} F. Peherstorfer, R. Steinbauer, {\em Orthogonal and
$L_q$- extremal polynomials on inverse images of polynomial mappings},
J. Comput. Appl. Math. 127 (2001), 297-315.

\bibitem{Ran} T. Ransford, {\em Potential theory in the complex plane}, Cambr.
Univ. Press., 1995.

\bibitem{Tot} V. Totik, {\em Polynomial inverse images and polynomial inequalities},
Acta Math. (Scandinavian) 187 (2001), 139-160.

\bibitem{Tot2} V. Totik, {\em Chebyshev constants and the inheritance problem},
to appear in J. Approx. Theory.

\bibitem{Wid} H. Widom, {\em Extremal polynomials associated with a system of
curves in the complex plane}, Adv. Math. 3 (1969), 127-232.

\end{thebibliography}
\end{document}